\documentclass[12pt]{amsart}
\usepackage{amssymb}
\usepackage{hyperref}
\usepackage{cases}
\usepackage{amsmath,latexsym,amssymb,amsthm,array,amsfonts}
\usepackage{mathrsfs,dsfont,bbm}
\usepackage{graphicx}

\theoremstyle{plain}

\numberwithin{equation}{section}

\theoremstyle{plain}

\newtheorem{theorem}{Theorem}[section]

\newtheorem{lemma}{Lemma}[section]

\begin{document}

\title[Approximating the Rosenblatt sheet by multiple Wiener integrals]
{An optimal approximation of Rosenblatt sheet by multiple Wiener integrals${}^{*}$}

\author[G. Shen,  Q. Yu ]
{Guangjun Shen,     Qian Yu }


\footnote[0]{${}^{*}$ This research is supported by the National
Natural Science Foundation of China (11271020).
  }
\date{}

\maketitle

\date{}

\begin{center}
{\it  Department of Mathematics, Anhui Normal University,
 Wuhu 241000, China.}
\end{center}

\maketitle

\begin{abstract}
Let $Z^{\alpha,\beta}$ be the Rosenblatt sheet with the representation
$$
Z^{\alpha,\beta}(t,s)=\int^t_0\int^s_0\int^t_0\int^s_0Q^\alpha(t,y_1,y_2)Q^\beta(s,u_1,u_2)B(dy_1,du_1)B(dy_2,du_2)
$$
where $B$ is a   Brownian sheet, $\frac12<\alpha,\beta<1$, $Q^\alpha$ and $Q^\beta$ are the given kernel.  In this paper,  we contruct  multiple Wiener integrals of the form
\begin{align*}
\int^t_0\int^s_0\int^t_0\int^s_0&[k_1(y_1,y_2)^{-\frac12\alpha}(u_1,u_2)^{-\frac12\beta}+k_2(y_1\vee y_2)^{\frac12\alpha}(y_1\wedge y_2)^{-\frac12\alpha}|y_1-y_2|^{\alpha-1}\\
&\cdot(u_1\vee u_2)^{\frac12\beta}(u_1\wedge u_2)^{-\frac12\beta}|u_1-u_2|^{\beta-1}]B(dy_1,du_1)B(dy_2,du_2),~~k_1,k_2\geq0,
\end{align*}
and obtain  an optimal approximation of $Z^{\alpha,\beta}(t,s)$.
\end{abstract}

\section{Introduction}\label{sec1}
Self-similar processes are stochastic processes that are invariant in distribution under
a suitable scaling of time and space. This property is crucial in applications
such as network traffic analysis, mathematical finance, astrophysics, hydrology and
image processing. For this reason, their analysis has long constituted an important
research direction in probability theory. The Hermite process  is an interesting class of self-similar processes with long
range dependence, it is given as limits of the so called {\it
Non-Central Limit Theorem}~ studied in Dobrushin and Major
\cite{DM}, Taqqu \cite{Taqqu}. Let us briefly recall the general
context.

Denote by  $H_j(x)$  the Hermite polynomial of order $j$ defined by
$$H_j(x)=(-1)^je^{\frac{x^2}{2}}\frac{d^j}{dx^j}e^{\frac{-x^2}{2}},  \quad j=1,2,...$$
with $H_0(x)=1$, and let
the Borel function $g:\mathbb{R}\rightarrow\mathbb{R}$ satisfy $E(g(\xi_0))=0,E(g(\xi_0)^2)<\infty$ and
$$g(x)=\sum^{\infty}_{j=0}c_{j}H_{j}(x),  \quad  c_{j}=\frac{1}{j!}E[g(\xi_0)H_{j}(\xi_0)].$$
The Hermite rank of $ g$ is defined by
$$k=\min\{j:c_j\neq0\},$$
Clearly, $k\geq1$ since $E[g(\xi_0)]=0$.

Let $g$ be a function of Hermite rank $k$ and let  $(\xi_n)_{n\in\mathbb{N}}$ be a stationary centered Gaussian sequence with $E(\xi^2_{n})=1$ which exhibits long range dependence in the sense that the correlation function satisfies
 \begin{equation}\label{sec1-eq1.1}
 r(n):=E(\xi_{0}\xi_{n})=n^{\frac{2H-2}{k}}L(n),
\end{equation}
where $k\geq1$ is an integer, $H\in(\frac12,1)$ and $L$ is a slowly varying function at infinity. Then,
 the \emph{Non Central Limit Theorem} implies that the stochastic processes
$$\frac{1}{n^H}\sum^{[nt]}_{j=1}g(\xi_j),  $$
converges, as $n\rightarrow\infty$, in the sense of finite dimensional distributions to the process
\begin{equation}\label{sec1-eq1.1}
Z_H^k(t)=c(H,k)\int_{\mathbb{R}^k}\int^t_0\left(\prod^k_{j=1}(s-y_j)_{+}^
{-(\frac12+\frac{1-H}{k})}\right)dsdB(y_1)...dB(y_k),
\end{equation}
 where $x_+=\max\{x,0\}$ and the above integral is a Wiener-It\^{o} multiple
integral with respect to the standard Brownian motion
$(B(y))_{y\in\mathbb{R}}$ excluding the diagonals $\{y_i=y_j\},
i\neq j$, $c(H,k)$ is a positive normalization constant depending
only on $H$ and $k$  such that $E(Z_H^k(1))^2=1$. The process
$(Z_H^k(t))_{t\geq 0}$ is called as {\it the Hermite process} of order $k$, it is
$H$ self-similar   and   has stationary increments.
The class of Hermite processes includes fractional Brownian motion ($k=1$)
which is the only Gaussian process in this class.  Their practical aspects are striking:
they provide a wide class of processes from which to model long memory, selfsimilarity,
and H\"{o}lder-regularity, allowing significant deviation from fractional Brownian motion and other
Gaussian processes. Since they are non-Gaussian and self-similar with stationary increments,
the Hermite processes can also be an input in models where self-similarity
is observed in empirical data which appears to be non-Gaussian.
For $k\geq2$,  the process \eqref{sec1-eq1.1} is not Gaussian. When
$k=2$, the process  \eqref{sec1-eq1.1} is known as the Rosenblatt
process (see Taqqu \cite{Ta4}).
More works for the Hermite process and Rosenblatt process can be found in  Albin \cite{Albin}, Leonenko and Ahn \cite{Leonenko},  Abry and Pipiras \cite{Abry},
Maejima and Tudor \cite{MT}, Tudor \cite{Tu5},  Chronopoulou {\it et al}. \cite{Tu2}, Tudor and Viens
\cite{Tu6}, Torres and  Tudor \cite{Tu4},  Shieh and Xiao \cite{Sx},  Pipiras and Taqqu \cite{PT}, Maejima
and Tudor \cite{MT1, MT2}, Chen,  Sun and  Yan \cite{Ch}, Garz\'{o}n,  Torres and Tudor \cite{gtt},
 Tudor \cite{Tu8}, Yan, Li and Wu \cite{C18}, Shen, Yin and Zhu \cite{shen}  and the reference therein.

Motivated by all these results, in this paper,  we will prove the optimal approximation theorem of Rosenblatt sheet based on the multiple Wiener integrals
of form
\begin{align*}
\int^t_0\int^s_0\int^t_0\int^s_0&[k_1(y_1,y_2)^{-\frac12\alpha}(u_1,u_2)^{-\frac12\beta}+k_2(y_1\vee y_2)^{\frac12\alpha}(y_1\wedge y_2)^{-\frac12\alpha}|y_1-y_2|^{\alpha-1}\\
&\cdot(u_1\vee u_2)^{\frac12\beta}(u_1\wedge u_2)^{-\frac12\beta}|u_1-u_2|^{\beta-1}]B(dy_1,du_1)B(dy_2,du_2),~t,s>0
\end{align*}
with $k_1,k_2>0$. Recall that  Rosenblatt sheet with parameter $\frac12<\alpha, \beta<1$ admits an integral
representation of the form (see Tudor \cite{Tu8}), for $t\in[0,T],s\in[0,S]$
\begin{equation*}
\begin{split}
Z^{\alpha,\beta}(t,s)&=d_\alpha d_\beta\int^t_0\int^s_0\int^t_0\int^s_0\int^t_{y_1\vee
y_2}\frac{\partial K^{\alpha'}}{\partial x}(x,y_1)
\frac{\partial K^{\alpha'}}{\partial x}(x,y_2)dx\\
&\quad \times\int^s_{u_1\vee u_2}\frac{\partial K^{\beta'}}{\partial
y}(y,u_1)
\frac{\partial K^{\beta'}}{\partial y}(y,u_2)dyB(dy_1,du_1)B(dy_2,du_2),\\
&:=\int^t_0\int^s_0\int^t_0\int^s_0
Q_\alpha(t,y_1,y_2)Q_\beta(s,u_1,u_2)B(dy_1,du_1)B(dy_2,du_2)\\
\end{split}
\end{equation*}
where $B$ is a standard Brownian sheet and $K$ is the deterministic kernel given by
\begin{equation}\label{sec1-eq1.4}
K^H(t,s)=c_Hs^{\frac12-H}\int^t_s(u-s)^{H-\frac32}u^{H-\frac12}du  \quad \textmd{for} \quad t>s,
\end{equation}
with $c_H=\sqrt{\frac{H(2H-1)}{B(2-2H,H-\frac12)}}$, $B(\cdot,\cdot)$ represents the Beta function,
$H'=\frac{H+1}{2}$ and $d_H=\frac{1}{H+1}\sqrt{\frac{H}{2(2H-1)}}$
 is a normalizing
constant. Denote
$$Q_H(t,y_1,y_2)=d_H1_{[0,t]}(y_1)1_{[0,t]}(y_2)\int^t_{y_1\vee y_2}
\frac{\partial K^{H'}}{\partial x}(x,y_1)
\frac{\partial K^{H'}}{\partial x}(x,y_2)dx.$$

In general, for every Borel measurable
function $\zeta\in L^2([0,T]^2\times[0,S]^2)$ the stochastic integral
$$M_{t,s}(\zeta):=\int^t_0\int^s_0\int^t_0\int^s_0\zeta(y_1,y_2,u_1,u_2)B(dy_1,du_1)B(dy_2,du_2),t\in[0,T],s\in[0,S]$$
is well-defined, and the optimal approximation problem is to estimate
\begin{equation}\label{sec1-eq1.5}
\inf_{\zeta\in L^2([0,T]^2\times[0,S]^2)}\sup_{t\in[0,T], s\in[0,S]}E(Z^{\alpha,\beta}(t,s)-M_{t,s}(\zeta))^2.
\end{equation}
Noting that if the above minimum is attained at the function $\zeta^*$, then
$\zeta^*>0$ a.e..  In fact, we have
\begin{align*}
E&[Z^{\alpha,\beta}(t,s)-M_{t,s}(\zeta)]^2=\frac12 t^{2\alpha}s^{2\beta}
+2\int^t_0\int^s_0\int^t_0\int^s_0\zeta^2(y_1,y_2,u_1,u_2)du_2dy_2du_1dy_1\\
&\quad-4\int^t_0\int^s_0\int^t_0\int^s_0 Q^\alpha(t,y_1,y_2)Q^\beta(s,u_1,u_2)\zeta(y_1,y_2,u_1,u_2)du_2dy_2du_1dy_1\\
\end{align*}
for all $t,s\geq0$.

 If $\zeta^*(y_1,y_2,u_1,u_2)\leq0$, then
$$\sup_{t\in[0,T],s\in[0,S]}E[Z^{\alpha,\beta}(t,s)-M_{t,s}(\zeta^*)]^2\geq
\sup_{t\in[0,T],s\in[0,S]}E[Z^{\alpha,\beta}(t,s)-M_{t,s}(|\zeta^*|)]^2.$$
This gives  the contradiction. Hence, we can assume that $k_1,k_2>0$  and study the
optimal approximation problem
\begin{equation}\label{sec1-eq1.6}
\inf_{\zeta\in\mathcal{K}}\sup_{t\in[0,T],s\in[0,S]}E[Z^{\alpha,\beta}(t,s)-M_{t,s}(\zeta)]^2
\end{equation}
where
\begin{align*}
\mathcal{K}=\{\zeta(y_1,y_2,u_1,u_2)=&k_1(y_1,y_2)^{-\frac12\alpha}(u_1,u_2)^{-\frac12\beta}+k_2(y_1\vee y_2)^{\frac12\alpha}(y_1\wedge y_2)^{-\frac12\alpha}|y_1-y_2|^{\alpha-1}\\
&\quad\cdot(u_1\vee u_2)^{\frac12\beta}(u_1\wedge u_2)^{-\frac12\beta}|u_1-u_2|^{\beta-1},k_1,k_2>0\},
\end{align*}
since $Q_H(t,y_1,y_2)\leq C_{H,T}\{(y_1,y_2)^{-\frac12H}+(y_1\vee y_2)^{\frac12H}(y_1\wedge y_2)^{-\frac12H}|y_1-y_2|^{H-1}\}.$
For $k\in \mathcal{K},$ denote
$$f(t,s,k_1,k_2):=2E[Z^{\alpha,\beta}(t,s)-M_{t,s}(\zeta)]^2, t,s \geq 0.$$
 The similar approximation for the fractional Brownian motion and Rosenblatt process are first considered by Banna and Mishura \cite{BM2008}, Mishura and Banna \cite{MB2009} and Yan, Li and Wu \cite{C18}, respectively.

The rest of this paper is organized as follows.  Section 2 give the representation of
the function $f(t,s,k_1,k_2)=2E(Z^{\alpha,\beta}(t,s)-M_{t,s}(\zeta))^2$ for $\zeta\in\mathcal{K}$.
In Section 3, we consider the function $\sup_{t\in[0,T],s\in[0,S]} f(t,s,k_1,k_2)$ in the compact rectangle interval $[0,T]\times[0,S]$. In  Section 4 and Section 5, we consider the optimal approximation in the two
case $\Delta\leq0$ and $\Delta>0$, respectively. Two special cases be considered in Section 6.
\section{The representation of $f(t,s,k_1,k_2)$}\label{sec2}
In this section, we will  give the representation of $f(t,s,k_1,k_2)=2E[Z^{\alpha,\beta}(t,s)-M_{t,s}(\zeta)]^2$
for $\zeta\in\mathcal{K}$.

\begin{theorem}\label{th2.1}
Let
$$
a(k_2)=1+\frac{4k_2^2}{\alpha\beta}B(1-\alpha,2\alpha-1)B(1-\beta,2\beta-1)-\frac{8k_2}{\alpha\beta}C_2(\alpha)C_2(\beta),
$$
$$
b(k_2)=C_1(\alpha)C_1(\beta)-4k_2B(1-\alpha,\alpha)B(1-\beta,\beta),
$$
where $C_1(\alpha)=d_\alpha c^2_{\alpha'}B^2(1-\alpha,\frac12\alpha)$, and
$$C_2(\alpha)=d_\alpha c^2_{\alpha'}\int_0^1\int_0^{s}r^{-\alpha}(1-s)^{\frac12\alpha-1}(1-r)^{\frac12\alpha-1}(s-r)^{\alpha-1}drds,$$
for all $k_1,k_2\geq0$ and $\frac12<\alpha,\beta<1.$ Then we have
$$f(t,s,k_1,k_2)=a(k_2)t^{2\alpha}s^{2\beta}-8k_1b(k_2)ts+4k_1^2\frac{t^{2-2\alpha}s^{2-2\beta}}{(1-\alpha)^2(1-\beta)^2},~~ t\in[0,T],s\in[0,S].$$

{\bf Remark} 1. As an immediate result we have  $a(k_2)\geq0$ and
\begin{equation}\label{sec2-eq2.1}
b^2(k_2)(1-\alpha)^2(1-\beta)^2\leq\frac14a(k_2)\\
\end{equation}
for all $\alpha,\beta\in(\frac12,1)$ , since $f(t,s,k_1,k_2)\geq0$. Notice that $a(k_2)$ is
also a quadratic equation in $k_2$, we get
$$\frac{4}{\alpha\beta}C_2^2(\alpha)C_2^2(\beta)\leq B(1-\alpha,2\alpha-1)B(1-\beta,2\beta-1)$$
for all $\frac12<\alpha,\beta<1$.

2. Using the constant $C_1, C_2$ we give the main results and at the end of this paper we give the numerical of these constants (see Figure 1,2,3, 4.)
\end{theorem}
\begin{proof}
It is easy to calculate that
\begin{align*}
\int^t_0\int^s_0\int^t_0\int^s_0 &Q^\alpha(t,y_1,y_2)Q^\beta(s,u_1,u_2)(y_1y_2)^{-\frac12\alpha}(u_1u_2)^{-\frac12\beta}du_2dy_2du_1dy_1\\
&=\int^t_0\int^t_0Q^\alpha(t,y_1,y_2)(y_1,y_2)^{-\frac12\alpha}dy_1dy_2\int^s_0\int^s_0Q^\beta(s,u_1,u_2)(u_1,u_2)^{-\frac12\beta}du_1du_2\\
&=d_\alpha c^2_{\alpha'}\int^t_0\int^t_0\int^t_{y_1\vee y_2}(y_1y_2)^{-\alpha}u^\alpha(u-y_1)^{\frac12\alpha-1}(u-y_2)^{\frac12\alpha-1}dudy_1dy_2\\
&\quad\times d_\beta  c^2_{\beta'}\int^s_0\int^s_0\int^s_{u_1\vee u_2}(u_1u_2)^{-\beta}u^\beta(u-u_1)^{\frac12\beta-1}(u-u_2)^{\frac12\beta-1}dudy_1dy_2\\
&=d_\alpha c^2_{\alpha'}d_\beta c^2_{\beta'}B^2(1-\alpha,\frac12\alpha)B^2(1-\beta,\frac12\beta)ts=C_1(\alpha)C_1(\beta)ts,
\end{align*}
and
\begin{align*}
\int^t_0\int^s_0\int^t_0\int^s_0 &Q^\alpha(t,y_1,y_2)Q^\beta(s,u_1,u_2)(y_1\vee y_2)^{\frac12\alpha}(y_1\wedge y_2)^{-\frac12\alpha}|y_1-y_2|^{\alpha-1}\\
&\qquad\cdot(u_1\vee u_2)^{\frac12\beta}(u_1\wedge u_2)^{-\frac12\beta}|u_1-u_2|^{\beta-1}du_2dy_2du_1dy_1\\
&=d_\alpha c^2_{\alpha'}[\int^t_0\int^u_0\int^u_0(y_1y_2)^{-\frac12\alpha}u^\alpha(u-y_1)^{\frac12\alpha-1}(u-y_2)^{\frac12\alpha-1}\\
&\qquad\cdot(y_1\vee y_2)^{\frac12\alpha}(y_1\wedge y_2)^{-\frac12\alpha}|y_1-y_2|^{\alpha-1}dy_1dy_2du]\\
&\quad\times d_\beta c^2_{\beta'}[\int^s_0\int^u_0\int^u_0(u_1u_2)^{-\frac12\beta}u^\beta(u-u_1)^{\frac12\beta-1}(u-u_2)^{\frac12\beta-1}\\
&\qquad\cdot(u_1\vee u_2)^{\frac12\beta}(u_1\wedge u_2)^{-\frac12\beta}|u_1-u_2|^{\beta-1}du_1du_2du]\\
&=\frac{C_2(\alpha)}{\alpha}t^{2\alpha}\frac{C_2(\beta)}{\beta}t^{2\beta}.
\end{align*}
 Then for all $t\in[0,T],s\in[0,S]$,  we have,
\begin{align*}
\int^t_0\int^s_0\int^t_0\int^s_0 &Q^\alpha(t,y_1,y_2)Q^\beta(s,u_1,u_2)\zeta(y_1,y_2,u_1,u_2)du_2dy_2du_1dy_1\\
&=k_1C_1(\alpha)C_1(\beta)ts+k_2\frac{C_2(\alpha)C_2(\beta)}{\alpha\beta}t^{2\alpha}s^{2\beta}.
\end{align*}
On the other hand,
\begin{align*}
\int^t_0\int^s_0&\int^t_0\int^s_0\zeta^2(y_1,y_2,u_1,u_2)du_2dy_2du_1dy_1\\
&=\int^t_0\int^s_0\int^t_0\int^s_0k_1^2(y_1y_2)^{-\alpha}(u_1u_2)^{-\beta}du_2dy_2du_1dy_2\\
&\quad+\int^t_0\int^s_0\int^t_0\int^s_0k_2^2(y_1\vee y_2)^\alpha(y_1\wedge y_2)^{-\alpha}|y_1-y_2|^{2\alpha-2}\\
&\qquad\qquad\qquad\cdot(u_1\vee u_2)^\beta(y_1\wedge u_2)^{-\beta}|u_1-u_2|^{2\beta-2}du_2dy_2du_1dy_2\end{align*}\begin{align*}
&\quad+\int^t_0\int^s_0\int^t_0\int^s_02k_1k_2(y_1y_2)^{-\frac12\alpha}(y_1\vee y_2)^{\frac12\alpha}(y_1\wedge y_2)^{-\frac12\alpha}|y_1-y_2|^{\alpha-1}\\
&\qquad\qquad\qquad\cdot(u_1u_2)^{-\frac12\beta}(u_1\vee u_2)^{\frac12\beta}(u_1\wedge u_2)^{-\frac12\beta}|u_1-u_2|^{\beta-1}du_2dy_2du_1dy_2\\
&:=I_1+I_2+I_3,
\end{align*}
where
\begin{align*}
I_1&=k_1^2\int^t_0\int^t_0(y_1y_2)^{-\alpha}dy_1dy_2\int^s_0\int^s_0(u_1u_2)^{-\beta}du_1du_2\\
&=\frac{t^{2-2\alpha}s^{2-2\beta}}{(1-\alpha)^2(1-\beta)^2}k_1^2.
\end{align*}
\begin{align*}
I_2&=k_2^2(\int^t_0\int^t_0(y_1\vee y_2)^\alpha(y_1\wedge y_2)^{-\alpha}|y_1-y_2|^{2\alpha-2}dy_1dy_2)\\
&\qquad\qquad\cdot(\int^s_0\int^s_0(u_1\vee u_2)^\beta(y_1\wedge u_2)^{-\beta}|u_1-u_2|^{2\beta-2}du_1du_1)\\
&=k_2^2\left[\int^t_0\int^{y_2}_0y_2^\alpha y_1^{-\alpha}(y_2-y_1)^{2\alpha-2}dy_1dy_2+\int^t_0\int^t_{y_2}y_1^\alpha y_2^{-\alpha}(y_1-y_2)^{2\alpha-2}dy_1dy_2\right]\\
&\qquad\cdot\left[\int^s_0\int^{u_2}_0u_2^\beta u_1^{-\beta}(u_2-u_1)^{2\beta-2}du_1du_2+\int^u_0\int^s_{u_2}u_1^\beta u_2^{-\beta}(u_1-u_2)^{2\beta-2}du_1du_2\right]\\
&=k_2^2(2\int^t_0y_2^{2\alpha-1}dy_2\int^1_0u^{-\alpha}(1-u)^{2\alpha-2}du)\cdot(2\int^s_0u_2^{2\beta-1}du_2\int^1_0u^{-\beta}(1-u)^{2\beta-2}du)\\
&=\frac{k_2^2}{\alpha\beta}B(1-\alpha,2\alpha-1)B(1-\beta,2\beta-1)t^{2\alpha}s^{2\beta}.
\end{align*}
\begin{align*}
I_3&=2k_1k_2(\int^t_0\int^t_0(y_1y_2)^{-\frac12\alpha}(y_1\vee y_2)^{\frac12\alpha}(y_1\wedge y_2)^{-\frac12\alpha}|y_1-y_2|^{\alpha-1}dy_1dy_2)\\
&\qquad\cdot(\int^s_0\int^s_0(u_1u_2)^{-\frac12\beta}(u_1\vee u_2)^{\frac12\beta}(u_1\wedge u_2)^{-\frac12\beta}|u_1-u_2|^{\beta-1}du_1du_2)\\
&=2k_1k_2(\int^t_0\int^{y_2}_0y_1^{-\alpha}(y_2-y_1)^{\alpha-1}dy_1dy_2+\int^t_0\int_{y_2}^ty_2^{-\alpha}(y_1-y_2)^{\alpha-1}dy_1dy_2)\\
&\qquad\cdot(\int^s_0\int^{u_2}_0u_1^{-\beta}(u_2-u_1)^{\beta-1}du_1du_2+\int^s_0\int_{u_2}^su_2^{-\beta}(u_1-u_2)^{\beta-1}du_1du_2)\\
&=2k_1k_2(2\int^t_0dy_1\int^1_0u^{-\alpha}|1-u|^{\alpha-1}du)\cdot(2\int^s_0du_1\int^1_0u^{-\beta}|1-u|^{\beta-1}du)\\
&=8k_1k_2B(1-\alpha,\alpha)B(1-\beta,\beta)ts.
\end{align*}
Hence,
\begin{align*}
&\int^t_0\int^s_0\int^t_0\int^s_0\zeta^2(y_1,y_2,u_1,u_2)du_2dy_2du_1dy_1=\frac{k_1^2}{(1-\alpha)^2(1-\beta)^2}t^{2-2\alpha}s^{2-2\beta}\\
&\quad+\frac{k_2^2}{\alpha\beta}B(1-\alpha,2\alpha-1)B(1-\beta,2\beta-1)t^{2\alpha}s^{2\beta}+8k_1k_2B(1-\alpha,\alpha)B(1-\beta,\beta)ts,
\end{align*}
for all $\zeta\in\mathcal{K}$. It follows that
\begin{align*}
f(t,s,k_1,k_2)&=2E(Z^{\alpha,\beta}(t,s)-M_{t,s}(\zeta))^2\\
&=t^{2\alpha}s^{2\beta}+4\int^t_0\int^s_0\int^t_0\int^s_0\zeta^2(y_1,y_2,u_1,u_2)du_2dy_2du_1dy_1\\
&-8\int^t_0\int^s_0\int^t_0\int^s_0Q^\alpha(t,y_1,y_2)Q^\beta(s,u_1,u_2)\zeta(y_1,y_2,u_1,u_2)du_2dy_2du_1dy_1\\
&=a(k_2)t^{2\alpha}s^{2\beta}-8k_1b(k_2)ts+4k_1^2\frac{t^{2-2\alpha}s^{2-2\beta}}{(1-\alpha)^2(1-\beta)^2}.
\end{align*}
This completes the proof.\\
\end{proof}

\section{The maximum value of $f(t,s,k_1,k_2)$ }\label{sec3}
In this section, in order to obtain the optimal approximation with  $k_1,k_2>0$. We need to find the maximum value point
$P_0(t_0,s_0)$ in the open rectangle interval $(0,T)\times(0,S)$,  and  the maximum value point at the boundary of rectangle interval $[0,T]\times[0,S]$.
Thus, we can get the $\sup\limits_{t\in[0,T],s\in[0,S]}f(t,s,k_1,k_2).$

\begin{lemma}\label{lem 3.1}
The function $f(t,s,k_1.k_2)$ at the open rectangle interval $(0,T)\times(0,S)$ can't get the maximum value point
$P_0(t_0,s_0)$, $t_0\in(0,T),s_0\in(0,S)$.
\end{lemma}
\begin{proof}
If the function $f(t,s,k_1,k_2)$ have maximum value point $P_0(t_0,s_0)$, $t_0\in(0,T),s_0\in(0,S)$, then $P_0(t_0,s_0)$
must be the stagnation point of the function
$$(t,s)\mapsto f(t,s,k_1,k_2).$$
That is,
$$\frac{\partial f(t,s,k_1,k_2)}{\partial t}|_{P_0(t_0,s_0)}=0 , ~~~ \frac{\partial f(t,s,k_1,k_2)}{\partial s}|_{P_0(t_0,s_0)}=0,$$
Solving equations set:
\begin{align*}
\frac{\partial f(t,s,k_1,k_2)}{\partial t}&=2t^{1-2\alpha}s^{2-2\beta}(\alpha a(k_2)t^{4\alpha-2}s^{4\beta-2}\\
&\quad-4k_1b(k_2)t^{2\alpha-1}s^{2\beta-1}+\frac{4k_1^2}{(1-\alpha)(1-\beta)^2})=0, \\
\frac{\partial f(t,s,k_1,k_2)}{\partial s}&=2t^{2-2\alpha}s^{1-2\beta}(\beta a(k_2)t^{4\alpha-2}s^{4\beta-2}\\
&\quad-4k_1b(k_2)t^{2\alpha-1}s^{2\beta-1}+\frac{4k_1^2}{(1-\alpha)^2(1-\beta)})=0.
\end{align*}
Let $x=t^{2\alpha-1}s^{2\beta-1},$
 elementary calculation can obtain
$$x=\frac{k_1}{b(k_2)(1-\alpha)^2(1-\beta)^2},$$
and
$$
\sqrt{a(k_2)}=2b(k_2)(1-\alpha)(1-\beta).$$
Hence, we have
\begin{align*}
f(t_0,s_0,k_1,k_2)=t_0s_0[a(k_2)\frac{k_1}{b(k_2)(1-\alpha)^2(1-\beta)^2}-4k_1b(k_2)]=0.
\end{align*}
But $f(\varepsilon,\varepsilon,0,0)=\varepsilon^2>0$, for any $\varepsilon>0$.
That is contradiction. Therefore, the function $f(t,s,k_1,k_2)$ at the open rectangle interval$(0,T)\times(0,S)$ can't get the maximum value point
$P_0(t_0,s_0)$, $t_0\in(0,T),s_0\in(0,S)$.
\end{proof}

At the boundary $0\times[0,S]$, and $[0,T]\times0$, we have $f(t,s,k_1,k_2)=0$. So, we  need to consider the case of extremum value point on the boundary $T\times[0,S]$ and $[0,T]\times S$, with $k_1, k_2>0$. We only think about the case of the boundary $T\times[0,S]$. In the same way, we can get case of the boundary $[0,T]\times S$.

It follows from  Theorem 2.1, we know that
\begin{equation}\label{sec3-eq3.1}
f(T,s,k_1,k_2)=a(k_2)T^{2\alpha}s^{2\beta}-8k_1b(k_2)Ts+4k_1^2\frac{T^{2-2\alpha}s^{2-2\beta}}{(1-\alpha)^2(1-\beta)^2},
\end{equation}
for all $s\in[0,S]$. Differentiating \eqref{sec3-eq3.1} with respect to $s$ leads to
\begin{equation}\label{sec3-eq3.2}
f_s(T,s,k_1,k_2)=2\beta a(k_2)T^{2\alpha}s^{2\beta-1}-8k_1b(k_2)T+\frac{8k_1^2}{(1-\alpha)^2(1-\beta)}T^{2-2\alpha}s^{1-2\beta}.
\end{equation}

Let $f_s(T,s,k_1,k_2)=0$ and $x=k_1s^{1-2\beta}$, which implies that
\begin{equation}\label{sec3-eq3.3}
F(x):=\beta a(k_2)T^{2\alpha}-4b(k_2)Tx+\frac{4T^{2-2\alpha}x^2}{(1-\alpha)^2(1-\beta)}=0
\end{equation}
and the discriminant $\Delta$ of the quadratic function $F(x)$ is
$$\Delta=16T^2((b(k_2))^2-\frac{\beta a(k_2)}{(1-\alpha)^2(1-\beta)}).$$
If $\Delta\leq 0$, then
\begin{equation}\label{sec3-eq3.4}
\begin{split}
\sup_{s\in[0,S]}f(T,s,k_1,k_2)&=f(T,S,k_1,k_2)\\
&=T^{2\alpha}S^{2\beta}a(k_2)-8k_1b(k_2)TS+4k_1^2\frac{T^{2-2\alpha}S^{2-2\beta}}{(1-\alpha)^2(1-\beta)^2}.
\end{split}
\end{equation}
Using the same method, on the boundary $[0,T]\times S$, we have
$$
\sup_{t\in[0,T]}f(t,S,k_1,k_2)=f(T,S,k_1,k_2).
$$
If $\Delta>0$, then the equation \eqref{sec3-eq3.3} has two real roots as follows
$$x_1=\frac{(1-\alpha)^2(1-\beta)}{2T^{1-2\alpha}}(b(k_2)+\sqrt{b^2(k_2)-\frac{\beta a(k_2)}{(1-\alpha)^2(1-\beta)}}),$$
and
$$x_2=\frac{(1-\alpha)^2(1-\beta)}{2T^{1-2\alpha}}(b(k_2)-\sqrt{b^2(k_2)-\frac{\beta a(k_2)}{(1-\alpha)^2(1-\beta)}}),$$
which says $s_1=k_1^{\frac{1}{2\alpha-1}}x_1^{-\frac{1}{2\alpha-1}}$, and $s_2=k_1^{\frac{1}{2\alpha-1}}x_2^{-\frac{1}{2\alpha-1}}$  are two stagnation points of the function $s\mapsto f(T,s,k_1,k_2)$. Hence,
  $s_1 , s_2 $ are the points of local maximum and minimum, respectively,  since the monotonicity of the function $s\mapsto f(T,s,k_1,k_2)$. This implies that
\begin{equation}\label{sec3-eq3.5}
\begin{split}
&\sup_{s\in[0,S]}f(T,s,k_1,k_2)=f(T,S,k_1,k_2), \qquad \textmd{if} \quad s_1\geq S,\\
&\sup_{s\in[0,S]}f(T,s,k_1,k_2)=\max\{f(T,s_1,k_1,k_2),f(T,S,k_1,k_2)\}, \qquad \textmd{if} \quad s_1<S.
\end{split}
\end{equation}

Using the same method, on the boundary $[0,T]\times S$, we can find a $t_1$, such that
\begin{equation}\label{sec3-eq3.6}
\begin{split}
&\sup_{t\in[0,T]}f(t,S,k_1,k_2)=f(T,S,k_1,k_2),\qquad \textmd{if} \quad t_1\geq T,\\
&\sup_{t\in[0,T]}f(t,S,k_1,k_2)=\max\{f(t_1,S,k_1,k_2),f(T,S,k_1,k_2)\}, \qquad \textmd{if} \quad t_1<T.
\end{split}
\end{equation}

\section{The optimal approximation, case $\Delta\leq0$ }\label{sec4}

\begin{theorem}\label{th 4.1}
If $\Delta\leq0$, then we have
\begin{align*}
\inf_{\zeta\in\mathcal{K}}\sup_{T\times[0,S]}f(t,s,k_1,k_2)=T^{2\alpha}S^{2\beta}a(k^{*}_2)-8k^{*}_1b(k^{*}_2)TS+4(k^{*}_1)^2\frac{T^{2-2\alpha}S^{2-2\beta}}{(1-\alpha)^2(1-\beta)^2},
\end{align*}
where
\begin{align*}
\mathcal{K}=\{\zeta(y_1,y_2,u_1,u_2)=&k^*_1(y_1,y_2)^{-\frac12\alpha}(u_1,u_2)^{-\frac12\beta}+k^*_2(y_1\vee y_2)^{\frac12\alpha}(y_1\wedge y_2)^{-\frac12\alpha}|y_1-y_2|^{\alpha-1}\\
&\quad\cdot(u_1\vee u_2)^{\frac12\beta}(u_1\wedge u_2)^{-\frac12\beta}|u_1-u_2|^{\beta-1},k_1,k_2>0\}
\end{align*}
and $(k_1^*,k_2^*)$ is the stagnation point of the function
$$(k_1,k_2)\mapsto f(T,S,k_1,k_2),$$
here
\begin{equation}\label{sec4-eq4.1}
\begin{split}
k_1^*=&\frac{4B(1-\alpha,\alpha)B(1-\beta,\beta)C_2(\alpha)C_2(\beta)-B(1-\alpha,2\alpha-1)B(1-\beta,2\beta-1)C_1(\alpha)C_1(\beta)}{8\alpha \beta(1-\alpha)^2(1-\beta)^2B^2(1-\alpha,\alpha)B^2(1-\beta,\beta)-B(1-\alpha,2\alpha-1)B(1-\beta,2\beta-1)}\\
&\quad\times(1-\alpha)^2(1-\beta)^2T^{2\alpha-1}S^{2\beta-1},
\end{split}
\end{equation}

\begin{equation}\label{sec4-eq4.2}
k_2^*=\frac{C_2(\alpha)C_2(\beta)-4\alpha\beta(1-\alpha)^2(1-\beta)^2B(1-\alpha,\alpha)B(1-\beta,\beta)C_1(\alpha)C_1(\beta)}{B(1-\alpha,2\alpha-1)B(1-\beta,2\beta-1)-8\alpha\beta(1-\alpha)^2(1-\beta)^2B^2(1-\alpha,\alpha)B^2(1-\beta,\beta)}.
\end{equation}
\end{theorem}
\begin{proof}
It follows from equation \eqref{sec3-eq3.4}, when $\Delta\leq0$,
$$
\sup_{s\in[0,S]}f(T,s,k_1,k_2)=T^{2\alpha}S^{2\beta}a(k_2)-8k_1b(k_2)TS+4(k_1)^2\frac{T^{2-2\alpha}S^{2-2\beta}}{(1-\alpha)^2(1-\beta)^2}\\
$$
for all $k_1,k_2\geq0$. Let now $(k_1^*,k_2^*)$ be the stagnation point of the function
$$(k_1,k_2)\mapsto f(T,S,k_1,k_2).$$
An elementary calculation can obtain $k_1^*,k_2^*$ which can be denoted by \eqref{sec4-eq4.1} and \eqref{sec4-eq4.2}, and the Hessian matrix $\mathbf{H}$ as follows

\begin{displaymath}
\mathbf{H} =
\left( \begin{array}{ccc}
\frac{\partial^2f(T,S,k_1k_2)}{\partial k_1^2} & \frac{\partial^2f(T,S,k_1k_2)}{\partial k_1\partial k_2}  \\
\frac{\partial^2f(T,S,k_1k_2)}{\partial k_2\partial k_1}  & \frac{\partial^2f(T,S,k_1k_2)}{\partial k_2^2}  \\
\end{array} \right)
\end{displaymath}
and $$\frac{\partial^2f(T,S,k_1k_2)}{\partial k_1^2}=\frac{8T^{2-2\alpha}S^{2-2\beta}}{(1-\alpha)^2(1-\beta)^2}>0,$$
\begin{align*}
\frac{\partial^2f(T,S,k_1k_2)}{\partial k_1\partial k_2}&=\frac{\partial^2f(T,S,k_1k_2)}{\partial k_2\partial k_1}\\
&=32TSB(1-\alpha,\alpha)B(1-\beta,\beta),
\end{align*}
$$\frac{\partial^2f(T,S,k_1k_2)}{\partial k_2^2}=\frac{8T^{2\alpha}S^{2\beta}}{\alpha\beta}B(1-\alpha,2\alpha-1)B(1-\beta,2\beta-1).$$
So,
\begin{align*}
|\mathbf{H}|=64T^2S^2(\frac{B(1-\alpha,2\alpha-1)B(1-\beta,2\beta-1)}{\alpha(1-\alpha)^2\beta(1-\beta)^2}-16B^2(1-\alpha,\alpha)B^2(1-\beta,\beta))>0
\end{align*}
for all $\frac12<\alpha,\beta<1$, since $B^2(1-\alpha,\alpha)<\frac{B(1-\alpha,2\alpha-1)}{1-\alpha}$. This means that the minimal value of
$(k_1,k_2)\mapsto f(T,S,k_1,k_2)$ is achieved at point $(k_1^*,k_2^*)$. This completes the proof.
\end{proof}

\section{The optimal approximation, case $\Delta>0$ }\label{sec5}

\begin{lemma}\label{lem 5.1}
If $\Delta>0$, then we have
$$s_1(k_1^*,k_2^*)<S,$$
and
$$s_2(k_1^*,k_2^*)<S,$$
where $(k_1^*,k_2^*)$ is the stagnation point of the function
$(k_1,k_2)\mapsto \sup_{s\in[0,S]}f(T,s,k_1,k_2)$.\\
\end{lemma}
\begin{proof}
We split the proof in two steps.\\
Step one. It is easy to obtain that,
\begin{equation}\label{sec5-eq5.1}
\frac{2}{s_1^{2\beta-1}}>\frac{1}{s_1^{2\beta-1}}+\frac{1}{s_2^{2\beta-1}}=\frac{x_1+x_2}{k_1}=\frac{b(k_2)(1-\alpha)^2(1-\beta)}{T^{1-2\alpha}k_1}.
\end{equation}
since $s_2>s_1$.
Hence, when $k_1=k_1^*$ and $k_2=k_2^*$,
\begin{equation}\label{sec5-eq5.2}
\begin{split}
\frac{b(k_2^*)(1-\alpha)^2(1-\beta)}{T^{1-2\alpha}k^*_1}&=\frac{(C_1(\alpha)C_1(\beta)-4k^*_2B(1-\alpha,\alpha)B(1-\beta,\beta))(1-\alpha)^2(1-\beta)}{T^{1-2\alpha}k^*_1}\\
&=\frac{S^{1-2\beta}}{1-\beta}>\frac2{S^{2\beta-1}}
\end{split}
\end{equation}
for $\beta\in(\frac12,1)$. This proves that $s_1(k_1^*,k_2^*)<S$.\\
Step two.
\begin{align*}
&\frac{\partial f(T,s,k_1^*,k_2^*)}{\partial s}
=2\beta a(k_2^*)T^{2\alpha}s^{2\beta-1}-8k^*_1b(k^*_2)T+\frac{8(k^*_1)^2}{(1-\alpha)^2(1-\beta)}T^{2-2\alpha}s^{1-2\beta}\\
&=2\beta a(k_2^*)T^{2\alpha}s^{2\beta-1}-8\frac{b(k_2^*)(1-\alpha)^2(1-\beta)^2}{T^{1-2\alpha}s^{1-2\beta}}b(k^*_2)T+\frac{8(\frac{b(k_2^*)(1-\alpha)^2(1-\beta)^2}{T^{1-2\alpha}s^{1-2\beta}})^2}{(1-\alpha)^2(1-\beta)}T^{2-2\alpha}s^{1-2\beta}\\
&=2\beta T^{2\alpha}s^{2\beta-1}(a(k_2^*)-4b^2(k_2^*)(1-\alpha)^2(1-\beta)^2),
\end{align*}
since,$$
k^*_1=\frac{b(k_2^*)(1-\alpha)^2(1-\beta)^2}{T^{1-2\alpha}S^{1-2\beta}},~~~
\textmd{and}~~~ 4b^2(k_2^*)\leq\frac{a(k_2^*)}{(1-\alpha)^2(1-\beta)^2}.$$ Hence,
$$
\frac{\partial f(T,s,k_1^*,k_2^*)}{\partial s}|_{s=S}>0.
$$
Thus,
$S<s_1$ or $S>s_2$. By step one, we know $S>s_1$, so $s_2(k_1^*,k_2^*)<S$.
\end{proof}

\begin{lemma}\label{lem 5.2}
Denote $$h(k_1,k_2)=f(T,s_1,k_1,k_2)-f(T,S,k_1,k_2).$$
Then the equation $h(k_1,k_2)=0$ has two solutions $k_1^{'}$ and $\overline{k_1}$, which satisfy $0<k_1^{'}<\overline{k_1}$,
$\frac{\partial h}{\partial k_1}|_{k_1=k_1^{'}}>0$ and $\frac{\partial h}{\partial k_1}|_{k_1=\overline{k_1}}=0.$
\end{lemma}
\begin{proof}
By Theorem \ref{th2.1} , we have
\begin{align*}
f(T,s_1,k_1,k_2)&=T^{2\alpha}s_1^{2\beta}a(k_2)-8k_1b(k_2)Ts_1+4k_1^2\frac{T^{2-2\alpha}s_1^{2-2\beta}}{(1-\alpha)^2(1-\beta)^2}\\
&=T^{2\alpha}(\frac{k_1}{x_1})^{\frac{2\beta}{2\beta-1}}a(k_2)-8k_1b(k_2)T(\frac{k_1}{x_1})^{\frac1{2\beta-1}}+4k_1^2\frac{T^{2-2\alpha}(\frac{k_1}{x_1})^{\frac{2-2\beta}{2\beta-1}}}{(1-\alpha)^2(1-\beta)^2}\\
&:=k_1^{\frac{\beta}{\beta-\frac12}}\varphi(k_2)=k_1^{\frac{2\beta}{2\beta-1}}\varphi(k_2).
\end{align*}
Let $h(k_1,k_2)=0$, which implies that
\begin{equation}\label{sec5-eq5.3}
k_1^{\frac{2\beta}{2\beta-1}}\varphi(k_2)=T^{2\alpha}S^{2\beta}a(k_2)-8k_1b(k_2)TS+4k_1^2\frac{T^{2-2\alpha}S^{2-2\beta}}{(1-\alpha)^2(1-\beta)^2}=f(T,S,k_1,k_2).\\
\end{equation}
Differentiating \eqref{sec5-eq5.3} with respect to $k_1$
\begin{equation}\label{sec5-eq5.4}
\frac{2\beta}{2\beta-1}k_1^{\frac{1}{2\beta-1}}\varphi(k_2)=-8b(k_2)TS+8k_1\frac{T^{2-2\alpha}S^{2-2\beta}}{(1-\alpha)^2(1-\beta)^2}.\\
\end{equation}
multiplying by $\frac{2\beta}{(2\beta-1)k_1}$ on both side of \eqref{sec5-eq5.3} leads to
\begin{equation}\label{sec5-eq5.5}
\frac{2\beta}{2\beta-1}k_1^{\frac{1}{2\beta-1}}\varphi(k_2)=\frac{2\beta a(k_2)T^{2\alpha}S^{2\beta}}{(2\beta-1)k_1}-\frac{16\beta b(k_2)TS}{2\beta-1}+8\beta k_1\frac{T^{2-2\alpha}S^{2-2\beta}}{(2\beta-1)(1-\alpha)^2(1-\beta)^2}.\\
\end{equation}
It follows that
\begin{align*}
-8b(k_2)TS&+8k_1\frac{T^{2-2\alpha}S^{2-2\beta}}{(1-\alpha)^2(1-\beta)^2}\\&=\frac{2\beta a(k_2)T^{2\alpha}S^{2\beta}}{(2\beta-1)k_1}
-\frac{16\beta b(k_2)TS}{2\beta-1}+8\beta k_1\frac{T^{2-2\alpha}S^{2-2\beta}}{(2\beta-1)(1-\alpha)^2(1-\beta)^2}.\\
\end{align*}
This implies that
\begin{equation}\label{sec5-eq5.6}
(4\beta-4)T^{2-2\alpha}S^{2-2\beta}k_1^2+4b(k_2)TS(1-\alpha)^2(1-\beta)^2k_1-\beta a(k_2)(1-\alpha)^2(1-\beta)^2T^{2\alpha}S^{2\beta}=0.
\end{equation}
This is a quadratic equation in $k_1$ with the two roots
$$\overline{k_1}=\frac{b(k_2)(1-\alpha)^2(1-\beta)^2+\sqrt{D}}{2(1-\beta)T^{1-2\alpha}S^{1-2\beta}},$$
$$\underline{k_1}=\frac{b(k_2)(1-\alpha)^2(1-\beta)^2-\sqrt{D}}{2(1-\beta)T^{1-2\alpha}S^{1-2\beta}},$$
since
$$D=b^2(k_2)(1-\alpha)^4(1-\beta)^4-\beta a(k_2)(1-\beta)^3(1-\alpha)^2>0.$$
It is easy to check that $\overline{k_1}$ is the solution to the equation set
\begin{displaymath}
\left\{ \begin{array}{ll}
h(k_1,k_2)=0, \\
\frac{\partial h}{\partial k_1}(k_1,k_2)=0.
\end{array} \right.
\end{displaymath}
In the follows, we will prove that $\underline{k_1}$
is not the solution of the equation $h(k_1,k_2)=0$. In fact,
\begin{align*}
&h(\underline{k_1},k_2)=f(T,s_1,\underline{k_1},k_2)-f(T,S,\underline{k_1},k_2)\\
&=x_2^{\frac{2\beta}{2\beta-1}}\varphi(k_2)S^{2\beta}-a(k_2)T^{2\alpha}S^{2\beta}+8b(k_2)TS^{2\beta}x_2-\frac{4T^{2-2\alpha}}{(1-\alpha)^2(1-\beta)^2}x_2^2S^{2\beta}\\
&=S^{2\beta}(x_2^{\frac{2\beta}{2\beta-1}}\varphi(k_2)-a(k_2)T^{2\alpha}+8b(k_2)Tx_2-\frac{4T^{2-2\alpha}}{(1-\alpha)^2(1-\beta)^2}x_2^2)\\
&=\frac{2\beta-1}{1-\beta}S^{2\beta}((\frac{x_2}{x_1})^{\frac{2\beta}{2\beta-1}}(4b(k_2)Tx_1-a(k_2)T^{2\alpha})-(4b(k_2)Tx_2-a(k_2)T^{2\alpha})),\\
\end{align*}
since$$
\varphi(k_2)=a(k_2)T^{2\alpha}x_1^{-\frac{2\beta}{2\beta-1}}-8b(k_2)x_1^{\frac1{2\beta-1}}+\frac{4T^{2-2\alpha}}{(1-\alpha)^2(1-\beta)^2}x_1^{2\beta-2}.
$$

By \eqref{sec3-eq3.3}, we have
$$x_1+x_2=b(k_2)T^{2\alpha-1}(1-\alpha)^2(1-\beta),\qquad x_1x_2=\frac14\beta a(k_2)T^{4\alpha-2}(1-\alpha)^2(1-\beta),$$
since $x_1, x_2$ are the root of equation $\beta a(k_2)T^{2\alpha}-4b(k_2)Tx+\frac{4T^{2-2\alpha}x^2}{(1-\alpha)^2(1-\beta)}=0.$
Thus,
$$\frac{4b(k_2)Tx_1}{a(k_2)T^{2\alpha}}=\beta\frac{x_1+x_2}{x_2},\qquad \frac{4b(k_2)Tx_2}{a(k_2)T^{2\alpha}}=\beta\frac{x_1+x_2}{x_1}.$$
Let $x=\frac{x_2}{x_1}\in(0,1)$. Then
\begin{align*}
(&\frac{x_2}{x_1})^{\frac{2\beta}{2\beta-1}}(4b(k_2)Tx_1-a(k_2)T^{2\alpha})-(4b(k_2)Tx_2-a(k_2)T^{2\alpha})\\
&=a(k_2)T^{2\alpha}((\frac{x_2}{x_1})^{\frac{2\beta}{2\beta-1}}(\frac{\beta(x_1+x_2)}{x_2}-1)-(\frac{\beta(x_1+x_2)}{x_1}-1))\\
&=a(k_2)T^{2\alpha}(x^{\frac{2\beta}{2\beta-1}}(\beta x^{-1}+\beta-1)-\beta-\beta x+1)\\
&:=a(k_2)T^{2\alpha}\phi(x).
\end{align*}
It is easy to check $\phi(x)>0$. In fact, $\phi(0)=1-\beta$, $\phi(1)=0$, and $\phi^{'}(0)=-\beta$, $\phi^{'}(1)=0$
$$\phi^{''}(x)=\frac{2\beta(1-\beta)}{(2\beta-1)^2}x^{\frac{2-2\beta}{2\beta-1}}(x^{\frac{1-2\beta}{2\beta-1}}-1)>0$$
for all $x\in(0,1)$ since $2\beta>1$. This shows that the function $\phi(x)$ is convex on $(0,1)$ and $\phi^{'}(x)$
is increasing strictly on $(0,1)$, which implies $\phi^{'}(x)<0$. It follows that $\phi^{'}(x)$ is decreasing strictly on $(0,1)$ and
$$\phi(x)>\phi(1)=0$$
for all $x\in(0,1)$, Thus, $h(\underline{k_1},k_2)>0$.\\
On the other hand, $h(0,k_2)=-a(k_2)T^{2\alpha}S^{2\beta}<0$, it follows that the equation $$h(k_1,k_2)=f(T,s_1,k_1,k_2)-f(T,S,k_1,k_2)=0$$ admits a root, denoted by $k_1^{'}$, on $(0,\underline{k_1})$.
Noting that the $k_1\mapsto f(T,s_1,k_1,k_2)$ is convex and increasing, we find that the equation
$$h(k_1,k_2)=f(T,s_1,k_1,k_2)-f(T,S,k_1,k_2)=0$$
 admits two roots at most since the function $k_1\mapsto f(T,S,k_1,k_2)$
is a quadratic function. Thus, $k_1^{'}$ is unique in $(0,\underline{k_1})$ and
$\frac{\partial h}{\partial k_1}|_{k_1=k_1^{'}}>0$.
\end{proof}

Now, we consider $\sup_{s\in[0,S]}f(T,s,k_1,k_2)$ at the case $\Delta>0$.

\begin{theorem}\label{th 5.3}
If $\Delta>0$, $k_1^{'}$ is given in Lemma 5.2. We have
$$\max\{f(T,s_1,k_1,k_2),f(T,S,k_1,k_2)\}=f(T,S,k_1^{'},k_2).$$
\end{theorem}
\begin{proof}
By Lemma \ref{lem 5.1}, we have
$$\sup_{s\in[0,S]}f(T,s,k_1,k_2)=\max\{f(T,s_1,k_1,k_2),f(T,S,k_1,k_2)\}.$$
It follows from Lemma \ref{lem 5.2} that
$$\max\{f(T,s_1,k_1,k_2),f(T,S,k_1,k_2)\}=f(T,s_1,k_1,k_2)1_{\{k_1>k_1^{'}\}}+f(T,S,k_1,k_2)_{\{k_1<k_1^{'}\}},$$
which implies that $$\max\{f(T,s_1,k_1,k_2),f(T,S,k_1,k_2)\}=f(T,S,k_1^{'},k_2),$$
since $k_1\mapsto f(T,s_1,k_1,k_2)$ is increasing and $f(T,S,k_1,k_2)$ is decreasing for $k_1<k_1^{'}$.
\end{proof}
{\bf Remark } 3. By Theorem \ref{th 5.3} and \eqref{sec3-eq3.5}, we have
$$\sup_{s\in[0,S]}f(T,s,k_1,k_2)=f(T,S,k_1^{'},k_2),$$
where $k_1^{'}\in(0,\underline{k_1})$.\\
Use the same method, on the boundary $[0,T]\times S$, there also exists  a $k_1^{''}>0$, such that
$$\sup_{t\in[0,T]}f(t,S,k_1,k_2)=f(T,S,k_1^{''},k_2).$$
From all above, we have
$$\sup_{t\in[0,T],s\in[0,S]}f(t,s,k_1,k_2)=\max\{f(T,S,k_1^{'},k_2),f(T,S,k_1^{''},k_2)\}$$

\begin{theorem}\label{th 5.4}
1)If $f(T,S,k_1^{'},k_2)>f(T,S,k_1^{''},k_2)$. Then minimal value
$$\inf_{\zeta\in\mathcal{K}}\sup_{t\in[0,T],s\in[0,S]}f(t,s,k_1,k_2)$$
is achieved at point $(T,S,k_1^{'},k_2^{'})$ and  equals to $f(T,S,k_1^{'},k_2^{'})$.\\
2)If $f(T,S,k_1^{'},k_2)<f(T,S,k_1^{''},k_2)$. Then minimal value $$\inf_{\zeta\in\mathcal{K}}\sup_{t\in[0,T],s\in[0,S]}f(t,s,k_1,k_2)$$ is achieved at point $(T,S,k_1^{''},k_2^{''})$ and  equals to $f(T,S,k_1^{''},k_2^{''})$.
\end{theorem}

\begin{proof}
1). If $f(T,S,k_1^{'},k_2)>f(T,S,k_1^{''},k_2)$, then
\begin{align*}
\sup_{t\in[0,T],s\in[0,S]}&f(t,s,k_1,k_2)=f(T,S,k_1^{'},k_2)\\
&=T^{2\alpha}S^{2\beta}a(k_2)-8k^{'}_1b(k_2)TS+4(k^{'}_1)^2\frac{T^{2-2\alpha}S^{2-2\beta}}{(1-\alpha)^2(1-\beta)^2}\\
&=(\frac{4T^{2\alpha}S^{2\beta}}{\alpha\beta}B(1-\alpha,2\alpha-1)B(1-\beta,2\beta-1))k_2^2\\
&\quad+(32k_1^{'}TSB(1-\alpha,\alpha)B(1-\beta,\beta)-\frac{8C_2(\alpha)C_2(\beta)}{\alpha\beta}T^{2\alpha}S^{2\beta})k_2\\
&\quad+T^{2\alpha}S^{2\beta}-8k_1^{'}TSC_1(\alpha)C_1(\beta)+\frac{4(k_1^{'})^2}{(1-\alpha)^2(1-\beta)^2}T^{2-2\alpha}S^{2-2\beta}.
\end{align*}
This is a quadratic equation of $k_2$, and $\frac{4T^{2\alpha}S^{2\beta}}{\alpha\beta}B(1-\alpha,2\alpha-1)B(1-\beta,2\beta-1)>0$. It is easy to find that, when
$$k^{'}_2=\frac{C_2(\alpha)C_2(\beta)-4\alpha\beta k_1^{'}T^{1-2\alpha}S^{1-2\beta}B(1-\alpha,\alpha)B(1-\beta,\beta)}{B(1-\alpha,2\alpha-1)B(1-\beta,2\beta-1)},$$
we have
\begin{align*}
\inf_{\zeta\in\mathcal{K}}\sup_{t\in[0,T]s\in[0,S]}f(t,s,k_1,k_2)&=\inf_{\zeta\in\mathcal{K}}f(T,S,k^{'}_1,k_2)=f(T,S,k_1^{'},k_2^{'}).
\end{align*}
2). If $f(T,S,k_1^{'},k_2)<f(T,S,k_1^{''},k_2)$, we have
\begin{align*}
\inf_{\zeta\in\mathcal{K}}\sup_{t\in[0,T]s\in[0,S]}f(t,s,k_1,k_2)&=\inf_{\zeta\in\mathcal{K}}f(T,S,k^{''}_1,k_2)\\
&=f(T,S,k_1^{''},k_2^{''}),
\end{align*}
where
$$k^{''}_2=\frac{C_2(\alpha)C_2(\beta)-4\alpha\beta k_1^{''}T^{1-2\alpha}S^{1-2\beta}B(1-\alpha,\alpha)B(1-\beta,\beta)}{B(1-\alpha,2\alpha-1)B(1-\beta,2\beta-1)}.$$
This completes the proof.
\end{proof}

\section{Two special cases}\label{sec6}
In this section we consider two special classes of the approximation function $\zeta\in\mathcal{K}$
First, if $k_2=0$, we only consider the boundary $T\times[0,S]$.
\begin{theorem}\label{th 6.1}
Let $\mathcal{K}_1=\{\zeta(y_1,y_2,u_1,u_2)=k(y_1y_2)^{-\frac12\alpha}(u_1u_2)^{-\frac12\beta},k>0\}$.\\
(1)If $C^2_1(\alpha)C^2_1(\beta)-\frac{\beta}{(1-\beta)(1-\alpha)^2}\leq0$, then
$$\inf_{\zeta\in\mathcal{K}_2}\sup_{t\in[0,T],s\in[0,S]}2E(Z^{\alpha,\beta}(t,s)-M_{t,s}(\zeta))^2=T^{2\alpha}S^{2\beta}(1-4C^2_1(\alpha)C^2_1(\beta)(1-\alpha)^2(1-\beta)^2)
$$
with $k=\frac{C_1(\alpha)C_1(\beta)(1-\alpha)^2(1-\beta)^2}{T^{1-2\alpha}S^{1-2\beta}}$.\\
(2)If $C^2_1(\alpha)C^2_1(\beta)-\frac{\beta}{(1-\beta)(1-\alpha)^2}>0$, then
$$\inf_{\zeta\in\mathcal{K}_2}\sup_{t\in[0,T],s\in[0,S]}2E(Z^{\alpha,\beta}(t,s)-M_{t,s}(\zeta))^2=f(T,S,k_1^{'},0)
$$
where $\zeta(y_1,y_2,u_1,u_2)=k_1^{'}(y_1y_2)^{-\frac12\alpha}(u_1u_2)^{-\frac12\beta}$
and $k_1^{'}$ is the smallest root of the equation $f(T,s_1,k,0)-f(T,S,k,0)=0$.
\end{theorem}
\begin{proof}
For $\zeta\in\mathcal{K}_1$ we have
\begin{align*}
2E(Z^{\alpha,\beta}(t,s)-M_{t,s}(\zeta))^2&=f(t,s,k,0)\\
&=t^{2\alpha}s^{2\beta}-8k_1C_1(\alpha)C_1(\beta)ts+\frac{4(k_1)^2}{(1-\alpha)^2(1-\beta)^2}t^{2-2\alpha}s^{2-2\beta}
\end{align*}
and $\Delta_1=64k^2T^2[C^2_1(\alpha)C^2_1(\beta)-\frac{\beta}{(1-\beta)(1-\alpha)^2}]$, which competes the proof.
\end{proof}

Second, we consider the case $k_1=0$.
\begin{theorem}\label{th 6.2}
Let
\begin{align*}
\mathcal{K}_2=\{\zeta(y_1,y_2,u_1,u_2)&=k(y_1\vee y_2)^{\frac12\alpha}(y_1\wedge y_2)^{-\frac12\alpha}|y_1-y_2|^{\alpha-1}\\
&\quad\cdot(u_1\vee u_2)^{\frac12\beta}(u_1\wedge u_2)^{-\frac12\beta}|u_1-u_2|^{\beta-1},k>0\}
\end{align*}
we have
\begin{align*}
\inf_{\zeta\in\mathcal{K}_2}&\sup_{t\in[0,T],s\in[0,S]}2E(Z^{\alpha,\beta}(t,s)-M_{t,s}(\zeta))^2\\&=(1-\frac{4C^2_2(\alpha)C^2_2(\beta)}{\alpha\beta B(1-\alpha,2\alpha-1)B(1-\beta,2\beta-1)})T^{2\alpha}S^{2\beta}
\end{align*}
with $k^{*}=\frac{C_2(\alpha)C_2(\beta)}{B(1-\alpha,2\alpha-1)B(1-\beta,2\beta-1)}$.
\end{theorem}

\begin{proof}
By Theorem \ref{th2.1}, we have
\begin{align*}
\inf_{\zeta\in\mathcal{K}_2}\sup_{t\in[0,T],s\in[0,S]}2E(Z^{\alpha,\beta}(t,s)-M_{t,s}(\zeta))^2
&=\inf_{\zeta\in\mathcal{K}_2}a(k)T^{2\alpha}S^{2\beta}.\\
\end{align*}
The function
$$a(k)=1+\frac{4k^2}{\alpha\beta}B(1-\alpha,2\alpha-1)B(1-\beta,2\beta-1)-\frac{8k}{\alpha\beta}C_2(\alpha)C_2(\beta)$$
is a quadratic equation in $k$, then
\begin{align*}
\inf_{\zeta\in\mathcal{K}_2}a(k)T^{2\alpha}S^{2\beta}&=a(k^{*})T^{2\alpha}S^{2\beta}\\
&=(1-\frac{4C^2_2(\alpha)C^2_2(\beta)}{\alpha\beta B(1-\alpha,2\alpha-1)B(1-\beta,2\beta-1)})T^{2\alpha}S^{2\beta}
\end{align*}
with $k^{*}=\frac{C_2(\alpha)C_2(\beta)}{B(1-\alpha,2\alpha-1)B(1-\beta,2\beta-1)}$. \\
and
\begin{align*}
\zeta(y_1,y_2,u_1,u_2)&=\frac{C_2(\alpha)C_2(\beta)(y_1\vee y_2)^{\frac12\alpha}(y_1\wedge y_2)^{-\frac12\alpha}|y_1-y_2|^{\alpha-1}}{B(1-\alpha,2\alpha-1)B(1-\beta,2\beta-1)}\\
&\quad\cdot(u_1\vee u_2)^{\frac12\beta}(u_1\wedge u_2)^{-\frac12\beta}|u_1-u_2|^{\beta-1}
\end{align*}
where $y_1,y_2>0$ and $u_1,u_2>0$. This completes the proof.
\end{proof}
{\bf\large Acknowledgements.}  The authors would  like to   thanks
Professor Litan Yan,  for
stimulating discussions.

\begin{center}
\begin{figure}[b]
  \includegraphics[width=0.6 \textwidth]{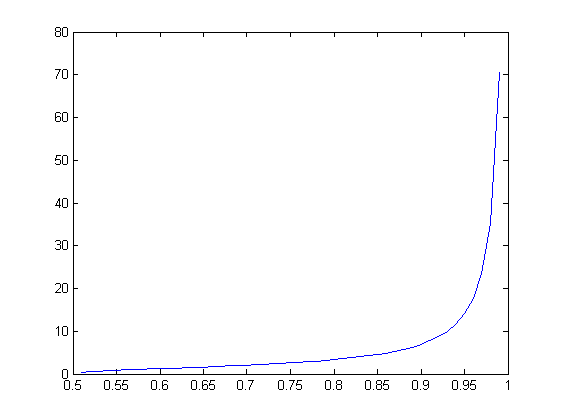}\\
  \caption{$\alpha\mapsto C_1(\alpha)$}\label{a}
\end{figure}
\end{center}
\begin{center}
\begin{figure}[b]
  \includegraphics[width=0.6 \textwidth]{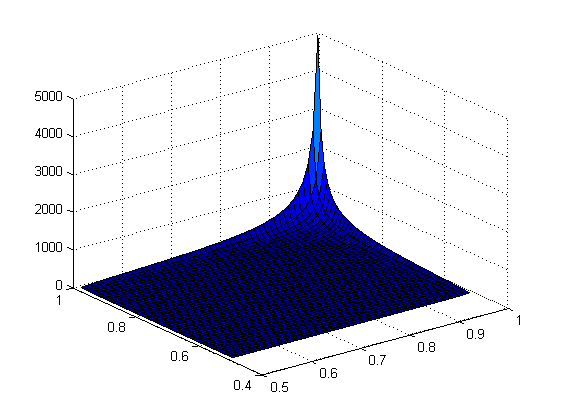}\\
  \caption{$(\alpha,\beta)\mapsto C_1(\alpha)C_1(\beta)$}\label{b}
\end{figure}
\end{center}
\begin{center}
\begin{figure}[b]
  \includegraphics[width=0.6 \textwidth]{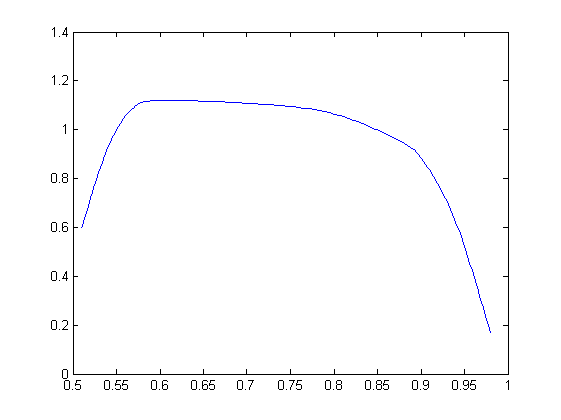}\\
  \caption{$\alpha\mapsto C_2(\alpha)$}\label{c}
\end{figure}
\end{center}
\begin{center}
\begin{figure}[b]
  \includegraphics[width=0.6 \textwidth]{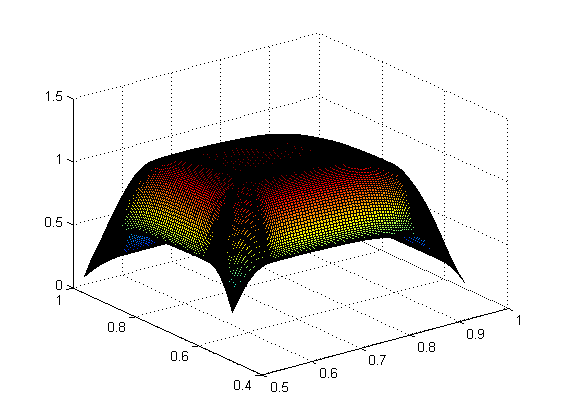}\\
  \caption{$(\alpha,\beta)\mapsto C_2(\alpha)C_2(\beta)$}\label{d}
\end{figure}
\end{center}
\end{document}